\documentclass[12pt]{amsart}
\usepackage{amssymb,enumerate,mathrsfs, amsmath, amsthm}
\usepackage{url,titletoc,enumitem}

\usepackage[usenames,dvipsnames]{xcolor}
\definecolor{darkblue}{rgb}{0.0, 0.0, 0.55}
\usepackage[pagebackref,colorlinks,linkcolor=BrickRed,citecolor=OliveGreen,urlcolor=darkblue,hypertexnames=true]{hyperref}
\usepackage{cmap}

\usepackage[utf8]{inputenc}
\usepackage{cleveref}
\crefname{section}{§}{§§}

\renewcommand{\qedsymbol}{\rule[.12ex]{1.2ex}{1.2ex}}
\renewcommand{\subset}{\subseteq}
\renewcommand{\emptyset}{\varnothing}

\newtheorem{theorem}{Theorem}[section]

\newtheorem{lemma}[theorem]{Lemma}

\newtheorem{prop}[theorem]{Proposition}

\newtheorem{alg}[theorem]{Algorithm}
\newtheorem{remark}[theorem]{Remark}

\newtheorem{thm}[theorem]{Theorem}
\newtheorem{lem}[theorem]{Lemma}

\newtheorem*{lemma*}{Lemma}

\def\beq{\begin{equation}}
\def\eeq{\end{equation}}

\numberwithin{equation}{section}

\def\beq{\begin{equation}}
\def\eeq{\end{equation}}

\def\cD{ {{\mathcal D}}}
\def\cH{ {{\mathcal H}}}

\def\bbN{ {\mathbb N}}
\def\bbS{ {\mathbb S}}
\def\bbK{ {\mathbb K}}
\def\K{ {\mathbb K}}

\def\R{ {\mathbb{R}} }
\def\C{ {\mathbb{C}} }

\def\bbS{{\mathbb S}}

\def\cA{ {\mathcal A} }

\def\cD{ {\mathcal D} }
\def\cE{ {\mathcal E} }

\def\cH{ {\mathcal H} }

\def\beq{\begin{equation}}
\def\eeq{\end{equation}}

\newcommand{\df}[1]{{\bf{#1}}{\index{#1}}}

\def\hbeta{{\hat{\beta}}}
\def\hsigma{{\hat{\sigma}}}
\def\hgamma{{\hat{\gamma}}}
\def\hY{{\hat{Y}}}
\def\fK{{\mathfrak{K}}}
\def\fE{{\mathfrak{E}}}

\def\abs{\partial^{\mathrm{abs}}}
\def\arv{\partial^{\mathrm{Arv}}}

\def\comat{\mathrm{co}^\mathrm{mat}}

\def\cDAC{{\cD_A^{\mathbb{C}}}}
\def\cDAR{{\cD_A^{\mathbb{R}}}}
\def\cDAK{{\cD_A^{\mathbb{K}}}}

\def\smnrg{{SM_n(\R)^g}}
\def\smmrg{{SM_m(\R)^g}}

\def\smnkg{{SM_n(\K)^g}}

\def\smlkg{{SM_\ell(\K)^g}}

\def\smnonerg{{SM_{n_1}(\R)^g}}
\def\smnonekg{{SM_{n_1}(\K)^g}}

\def\smntworg{{SM_{n_2}(\R)^g}}
\def\smntwokg{{SM_{n_2}(\K)^g}}

\def\smnthreekg{{SM_{n_3}(\K)^g}}

\def\smkg{{SM(\K)^g}}

\def\smdrg{{SM_d(\R)^g}}
\def\smdkg{{SM_d(\K)^g}}

\def\smspecialtwokg{{SM_{2n(g+1)}(\K)^g}}

\def\epsilon{\varepsilon}

\def\bem{\begin{pmatrix}}
\def\eem{\end{pmatrix}}

\title{Arveson Extreme Points Span Free Spectrahedra}

\author[E. Evert]{Eric Evert${}^1$}
\address{Eric Evert, Department of Mathematics\\
  University of California \\
  San Diego
   }
   \email{eevert@ucsd.edu}
\thanks{${}^1$Research supported by the NSF grant
DMS-1500835}
\author[J.W. Helton]{J. William Helton${}^1$}
\address{J. William Helton, Department of Mathematics\\
  University of California \\
  San Diego
   }
   \email{helton@math.ucsd.edu}

\makeindex

\date{\today}

\begin{document}
 
\begin{abstract}
Let $ \smnrg$ denote $g$-tuples of $n \times n$ real symmetric matrices. Given tuples $X=(X_1, \dots, X_g) \in \smnonerg$ and $Y=(Y_1, \dots, Y_g) \in \smntworg$, a matrix convex combination of $X$ and $Y$ is a sum of the form
\[
V_1^* XV_1+V_2^* Y V_2 \quad \quad \quad V_1^* V_1+V_2^* V_2=I_n
\]
where $V_1:\R^n \to \R^{n_1}$ and $V_2:\R^n \to \R^{n_2}$ are contractions. Matrix convex sets are sets which are closed under matrix convex combinations. A key feature of matrix convex combinations is that the $g$-tuples $X, Y$, and $V_1^* XV_1+V_2^* Y V_2$ do not need to have the same size. As a result, matrix convex sets are a dimension free analog of convex sets.

While in the classical setting there is only one notion of an extreme point, there are  three main notions of extreme points for matrix convex sets: ordinary, matrix, and absolute extreme points. Absolute extreme points are closely related to the classical Arveson boundary.
A central goal in the theory of matrix convex sets is to determine if one of these types of extreme points for a matrix convex set minimally recovers the set through matrix convex combinations. 

This article shows that every real compact matrix convex set which is defined by a linear matrix inequality is the matrix convex hull of its absolute extreme points, and that the absolute extreme points are the minimal set with this property. Furthermore, we give an algorithm which expresses a tuple as a matrix convex combination of absolute extreme points with optimal bounds. Similar results hold when working over the field of complex numbers rather than the reals.
\end{abstract}

\maketitle

\section{Introduction}

This paper concerns extreme points of
 noncommutative (free) convex sets.
In the free setting there are three major notions of an extreme point.
We shall study the most restricted class of extreme points,
the  absolute extreme points, a notion  introduced
 by Kleski \cite{KLS14}. This class of extreme points is closely related to Arveson's notion \cite{A69} of an irreducible boundary representation of an operator system \cite{KLS14,EHKM17}.
Hence the subject at hand goes back about 50 years.

Noncommutative convex sets can be described as solution
sets to types of linear matrix inequalities (LMIs), the workhorse of semidefinite programming.
Next we introduce this special type of LMI.
 Let $A =(A_1, A_2, \dots,A_g)$ be a $g$-tuple of bounded self-adjoint
operators on a real or complex Hilbert space $\cH$ and let $\tilde{\cH}$ be a real or complex Hilbert space with a nested sequences of subspaces $\{\tilde{\cH}^\ell\}_{\ell=1}^\infty$. We define an affine linear function $L_A$ on
tuples  $X =(X_1, X_2, \dots,X_g)$ of bounded self-adjoint
operators acting on $\tilde{\cH}^\ell$ for some $\ell$ by
\[
L_A (X)=I_{\cH}\otimes I_{\tilde{\cH}}+\Lambda_A (X)=I_{\cH}\otimes I_{\tilde{\cH}}+A_1 \otimes X_1 + \cdots + A_g \otimes X_g,
\]
and we define $\cD_A(\tilde{\cH}^\ell) $
to be the set of solutions to the LMI
\beq
\label{eq:GenMatConvLMI}
\cD_A(\tilde \cH) :=\{X \in \bbS(\tilde{\cH}^\ell)^g |\ L_A (X) \ pos \ semidef\ \}.
\eeq
Here $\bbS(\tilde{\cH}^\ell)^g$ denotes $g$-tuples of self-adjoint operators on $\tilde{\cH}^\ell$. The set
\[
\cup_{\ell=1}^\infty \cD_A (\tilde{\cH}^\ell)
\]
which we arrive at is a type of dimension free set which is operator convex and contains tuples of operators acting on each $\tilde{\cH}^\ell$.

A central question is whether
operator convex combinations of the absolute
extreme points of $\cup_{\ell=1}^\infty \cD_A (\tilde{\cH}^\ell)$ span $\cup_{\ell=1}^\infty \cD_A (\tilde{\cH}^\ell)$. We will define absolute extreme points (in a limited context) in Section \ref{sec:MatConvSets}. We remark that every closed matrix convex set can be expressed in the form of equation \eqref{eq:GenMatConvLMI} \cite{EW97}. Furthermore, matrix convex sets defined by noncommutative polynomial inequalities in matrix variables (``noncommutative semialgebraic sets") can be defined in this form where $\cH$ is finite dimensional \cite{HM12}.

Arveson conjectured that the irreducible boundary representations (in our language the absolute extreme points) span when $\cH$ and $\tilde \cH$ are Hilbert
spaces, see \cite{A69} and \cite{A72}. More on this viewpoint to extreme points is found in Section \ref{sec:AltContext}.
Many years later Dritschel and McCullough \cite{DM05}
showed if $\cH$ is separable and $\tilde \cH$ has cardinality
of the second uncountable ordinal, then uncountable combinations of absolute extreme points span. In that paper they say their dilation ideas were seriously influenced by a construction used in Agler's approach to model theory, see \cite{A88}. A decade later Davidson and Kennedy \cite{DK15} gave a complete and positive answer to Arveson's original question. As a consequence, \cite{DK15} shows that when $\cH$ and $\tilde{\cH}$ are both separable the absolute extreme points span. The finite dimensional version of the problem has been
  pursued for some time but until now has remained unsettled.

  In this paper we prove the finite dimensional version of Arveson's conjecture in the real and complex setting, see Theorem \ref{thm:AbsSpanBound}: \\
 {\it If $\cH = \R^d$ and
  $X$ is a $g$-tuple of self-adjoint $n \times n$ matrices over $\bbK=\R \mathrm{\ or \ } \C $ with $X$ in $\cD_A:= \cup_n \cD_A (\bbK^n)$,
  then   $X $ is a finite matrix convex combination of
 absolute extreme points of $\cD_A$ whose sum of sizes is bounded by $n(g+1)$ when $\bbK=\R$ and by $2n(g+1)$ when $\bbK=\C$.
  }
 \\ The proof is constructive and yields an algorithm
 for construction, see Section \ref{sec:Computation}.

 In the remainder of this section we introduce our basic definitions and notation and give a precise statement of our main results, Theorem \ref{thm:AbsSpanMin} and Theorem \ref{thm:AbsSpanBound}.  Some definitions just given will be repeated to provide a complete list.

\subsection{Notation and definitions}
Let $\bbK$ denote either $\R$ or $\C$. We will say a matrix is \df{self-adjoint} over $\bbK$ to mean the matrix is self-adjoint if $\bbK=\C$ or symmetric if $\bbK=\R$. For any positive integers $g$ and $n$, let $\smnkg$ denote the set of $g$-tuples $X=(X_1, \dots, X_g)$ of $n \times n$ self-adjoint matrices over $\bbK$ and let $\smkg$ denote the set $\smkg= \cup_n \smnkg$. Similarly, for positive integers $n, \ell$ and $g$ let $M_{n \times \ell} (\R)^g$ denote the set of $g$-tuples $\beta=(\beta_1, \dots, \beta_g)$ of $n \times \ell$ matrices over $\bbK$. Say a matrix $U \in M_n (\K)$ is a \df{unitary} if $U^* U=I_n$. Similarly, a matrix $V \in M_{n \times m} (\K)$ is an \df{isometry} if $V^* V=I_m$.

Given a matrix $M \in \K^{n \times n}$, a subspace $N \subset \K^n$ is a \df{reducing subspace} if both $N$ and $N^\perp$ are invariant subspaces of $M$. That is, $N$ is a reducing subspace for $M$ if $MN \subseteq N$ and $MN^\perp \subseteq N^\perp$. A tuple $X \in \smnkg$ is \df{irreducible} over $\K$ if the matrices $X_1, \dots, X_g$ have no common reducing subspaces in $\K^n$; a tuple is \df{reducible} over $\K$ if it is not irreducible over $\K$. 

Given a $g$-tuple $X \in \smnkg$  and a matrix $W \in M_{n \times m} (\bbK)$ we define the \df{conjugation of $X$ by $W$} by 
\[
W^* X W= (W^* X_1 W, \dots, W^* X_g W).
\]
If $W$ is a unitary (resp. isometry) then we say $W^* X W$ is a \df{unitary (resp. isometric) conjugation}. Given tuples $X,Y \in \smnkg$ say $X$ and $Y$ are \df{unitarily equivalent}, denoted by $X \sim_u Y$, if there exists a unitary matrix $U \in M_n (\bbK)$ such that
\[
U^* X U=Y.
\]
A subset $\Gamma \subset \smkg$ is \df{closed under unitary conjugation} if $X \in \Gamma$ and $Y \sim_u X$ implies $Y \in \Gamma$. We define the set \df{$\Gamma$ at level $n$}, denoted $\Gamma (n)$, by 
\[
\Gamma(n)= \Gamma \cap \smnkg.
\]
\index{$\Gamma (n)$}
That is, $\Gamma (n)$ is the set of $g$-tuples of $n \times n$ self-adjoint matrices in $\Gamma$.

\subsubsection{Matrix convex sets and extreme points} \label{sec:MatConvSets} Let $K \subset \smkg$.  A \df{matrix convex combination} of elements of $K$ is a finite sum of the form
\[
\sum_{i=1}^k V_i^* Y^i V_i \quad \quad \quad \quad \sum_{i=1}^k V_i^* V_i =I_n
\]
where $Y^i \in K(n_i)$ for $i=1, \dots, k$ and $V_i$ is an $n_i \times n$ matrix with entries in $\K$ for each $i$. If additionally $V_i \neq 0$ for each $i$, then the matrix convex combination is said to be \df{weakly proper}. If $K$ is closed under matrix convex combinations  then $K$ is \df{matrix convex}. 

Matrix convex combinations can equivalently be expressed via isometric conjugation. As before, let $\{Y^i\}_{i=1}^k \subset K$ be a finite collection of elements of $K$ and let $\{V_i\}_{i=1}^k$ be a collection of mappings from $\bbK^{n}$ to $\bbK^{n_i}$ such that  $\sum_{i=1}^k V_i^* V_i =I_n$. Define the $g$-tuple $Y$ and the isometry $V$ by 
\[
Y= \oplus_{i=1}^k Y^i \quad \quad \quad V^*=\begin{pmatrix} V_1^* & \cdots & V_k^* \end{pmatrix}.
\]
Then 
\beq
\label{eq:MatConvAreIsoConj}
V^* Y V= \sum_{i=1}^k V_i^* Y^i V_i \quad \quad \quad V^* V=\sum_{i=1}^k V_i^* V_i=I_n.
\eeq
In words, $V^* Y V$ is an isometric conjugation which is equal to the matrix convex combination $\sum_{i=1}^k V_i^* Y^i V_i$. A matrix convex combination of the form $V^* Y V$ is called a \df{compression} of $Y$. Given a set $K \subset \smkg$, define the \df{matrix convex hull} of $K$, denoted  
\[
\comat K,
\]
\index{$\comat K$}to be the smallest matrix convex set in $\smkg$ that contains $K$. Equivalently, $\comat K$ is the set of all matrix convex combinations of  elements of $K$.

Given a matrix convex set $K$, say $X \in K(n)$ is an \df{absolute extreme point} of $K$ if whenever $X$ is written as a weakly proper matrix convex combination $X=\sum_{i=1}^k V_i^* Y^i V_i$, then for all $i$ either $n_i=n$ and $X \sim_u Y^i$ or $n_i > n$ and there exists a tuple $Z^i \in K$ such that $X \oplus Z^i \sim_u Y^i$. We let $\abs K$ \index{$\abs K$}denote the set of absolute extreme points of $K$ and we call $\abs K$ the \df{absolute boundary} of $K$. We remark that an absolute extreme point $X=(X_1, \dots, X_g)$ has the property that $X_1, \dots, X_g$ is an irreducible collection of operators.

A matrix convex set $K$ is \df{bounded} if there is a real number $C>0$ such that 
\[
C-\sum_{i=1}^g X_i^2 \succeq 0
\]
for every tuple $X \in K$. We say $K$ is \df{closed} if $K(n)$ is closed for all $n \in N$ and we say $K$ is \df{compact} if $K$ is closed and bounded. We emphasize that $\comat K$ is not assumed to be closed.

\subsubsection{Free spectrahedra}

Free spectrahedra are a class of matrix convex sets; they are the solution set of a linear matrix inequality. 

Given a $g$-tuple $A$ of $d \times d$ self-adjoint matrices with entries in $\K$, let $\Lambda_A$ denote the \df{homogeneous linear pencil}
\[
\Lambda_A (x)=A_1 x_1+ \cdots + A_g x_g
\]
and let $L_A$ denote the \df{monic linear pencil}
\beq
\label{eq:MonicLinPencilDef}
L_A (x)=I_d + A_1 x_1+ \cdots +A_g x_g .
\eeq
Given a positive integer $n \in \bbN$ and an $X \in \smnkg$, the \df{evaluation} of the monic linear pencil $L_A$ on $X$ is defined by
\[
L_A (X)=I_{dn}+\Lambda_A (X)=I_{dn}+A_1 \otimes X_1 + \cdots + A_g \otimes X_g
\]
\index{$L_A(X)$}
where $\otimes$ denotes the Kronecker product.

The \df{free spectrahedron at level $n$}, denoted $\cD_A (\K^n)$, will typically be abbreviated 
\[
\cD^\K_A (n)=\{X \in \smnkg |\ L_A (X) \succeq 0 \}. \index{$\cD^\K_A (n)$}
\]
The corresponding \df{free spectrahedron} is the set $\cup_n \cD^\K_A(n) \subset \smkg$. In other words,
\[
\cD^\K_A=\{X \in \smkg |\ L_A (X) \succeq 0 \} \index{$\cD^\K_A$}.
\]
For emphasis, the elements of the \df{real free spectrahedron} $\cD_A^\R$ \index{$\cD^\R_A$} are $g$-tuples of real symmetric matrices, while the elements of the \df{complex free spectrahedron} $\cD_A^\C$ \index{$\cD^\C_A$} are $g$-tuples of complex self-adjoint matrices.

 We say a free spectrahedron $\cD_A^\K$ is \df{closed under complex conjugation} if $X \in \cD_A^\K$ implies 
\[
\overline{X}=(\overline{X}_1, \dots, \overline{X}_g) \in \cD_A^\K.
\]
Note that when $\K=\R$ the real free spectrahedron $\cDAR$ is trivially closed under complex conjugation. 
See \cite{HKM13}, \cite{Z17} and \cite{K+} for further discussion of linear pencils and free spectrahedra.

\subsection{Absolute extreme points span}

The following theorem, our first main result, shows that every compact free spectrahedron which is closed under complex conjugation is the matrix convex hull of its absolute extreme points. Furthermore, it shows that the absolute boundary is the smallest set of irreducible tuples which is closed under unitary conjugation and spans the free spectrahedron.
\begin{thm}
\label{thm:AbsSpanMin}
Assume $\K=\R$ or $\C$ and let $\cDAK$ be a compact free spectrahedron which is closed under complex conjugation. Then $\cDAK$ is the matrix convex hull of its absolute extreme points. In notation,
\[
\cDAK= \comat \abs \cDAK.
\]

Furthermore, if $K \subset \smkg$ is any closed matrix convex set and if $E \subset K$ is a set of irreducible tuples which is closed under unitary conjugation and whose matrix convex hull is equal to $K$, then $E$ contains the absolute boundary of $K$. In other words,
\[
K=\comat E \quad  \Rightarrow \quad  \abs K \subset E.
\]
In this sense the absolute extreme points are the minimal spanning set of a free spectrahedron.
\end{thm}
\begin{proof}
The fact that $\cDAK$ is the matrix convex hull of its absolute extreme points follows immediately from the forthcoming Theorem \ref{thm:AbsSpanBound}.

We now prove the second part of the result. $K \subset \smkg$ be any closed matrix convex set and let $E \subset \cDAK$ be a set of irreducible tuples which is closed under unitary conjugation and satisfies $\comat E=K.$ If $\abs K=\emptyset$ we are done. Otherwise, there is a positive integer $n$ and a tuple $X \in \smnkg$ such that $X \in \abs K (n)$. By assumption $X \in \comat E$, so there must exist a finite collection of tuples $\{Y^i\} \subset E$ and contractions $V_i: \K^n \to \K^{n_i}$ such that
\[
X=\sum_{i=1}^\mathrm{finite} V_i^* Y^i V_i.
\] 

Since $X$ is an absolute extreme point of $K$ and each $Y^i$ is irreducible we conclude that for each $i$ we have $n_i=n$ and there is a unitary $U_i:\K^n \to \K^n$ such that $U_i^* Y^i U_i=X$. By assumption $E$ is closed under unitary conjugation, so it follows that $X \in E$.
\end{proof}

\subsection{Dilations to Arveson extreme points}

Our second main result is a more quantitative version of Theorem \ref{thm:AbsSpanMin}.

\subsubsection{Dilations}

Let $K \subset \smkg$ be a matrix convex set and let $X \in K (n)$. If there exists a positive integer $\ell \in \bbN$ and $g$-tuples $\beta \in M_{n \times \ell} (\K)^g$ and $\gamma \in \smlkg$ such that
\[
Y=\begin{pmatrix}
X & \beta \\
\beta^* & \gamma
\end{pmatrix}= \left(\begin{pmatrix}
X_1 & \beta_1 \\
\beta^*_1 & \gamma_1
\end{pmatrix}, \cdots,
\begin{pmatrix}
X_g & \beta_g \\
\beta^*_g & \gamma_g 
\end{pmatrix} \right) \in K,
\]
then we say \df{$Y$ is an $\ell$-dilation of $X$}. The tuple $Y$ is said to be a \df{trivial dilation} of $X$ if $\beta=0$. Note that, if $V^*=\begin{pmatrix} I_n & 0 \end{pmatrix},$ then $X=V^* Y V$ with $V^* V=I_n$. That is, $X$ is a matrix convex combination of $Y$ in the spirit of equation \eqref{eq:MatConvAreIsoConj}.

Given tuples $A \in \smdkg$ and $X \in \smnkg$, we define the \df{dilation subspace of $\cDAK$ at $X$}, denoted $\fK^\K_{A,X}$, to be
\[
\fK^\K_{A,X}=\{\beta \in M_{n \times 1} (\K)^g | \ \ker L_A (X) \subset \ker \Lambda_A (\beta^*)\}. \index{$\fK^\K_{A,X}$}
\]
In this definition $\ker L_A(X)$ and $\ker \Lambda_A(X)$ are subspaces of $\K^{dn}$. The dilation subspace is examined in greater detail in Section \ref{sec:DiSubspace}.

\subsubsection{Arveson extreme points span}
The Arveson boundary of a matrix convex set $K$ is a classical dilation theoretic object which is closely related to the absolute boundary of $K$. We say a tuple $X \in K$ is an \df{Arveson extreme point} of $K$ if $K$ does not contain a nontrivial dilation of $X$. In other words, $X \in K$ is an Arveson extreme point of $K$ if and only if, if
\beq
\label{eq:ArvDefEq}
\begin{pmatrix}
X & \beta \\
\beta^* & \gamma 
\end{pmatrix} \in K(n+\ell)
\eeq
for some tuples $\beta \in M_{n \times \ell} (\K)^g$ and $\gamma \in \smlkg$, then $\beta=0$. A coordinate free definition is as follows. The point $X \in K(n)$ is an Arveson boundary point of $K$ if for each $m$ each  $Y$ in $K(m)$ and isometry $V:\K^n\to \K^m$ such that $X= V^*Y V$
 it follows that $VX=YV$. The set of Arveson extreme points of $K$, denoted by $\arv K, \index{$\arv K$}$ is called the \df{Arveson boundary} of $K$. If $Y$ is an Arveson extreme point of $K$ and $Y$ is an ($\ell$-)dilation of $X \in K$, then we will say $Y$ is an \df{Arveson ($\ell$-)dilation} of $X$.

The Arveson and absolute extreme points of a matrix convex set are closely related. Indeed the following theorem shows that a tuple is an absolute extreme point if and only if it is an irreducible Arveson extreme point.

\begin{theorem}
\label{theorem:EHKMRealComplex} Let $\cDAK$ be a free spectrahedron which is closed under complex conjugation. Then $X \in \cDAK$ is an absolute extreme point of $\cDAK$ if and only if $X$ is irreducible over $\K$ and $X$ is Arveson extreme point of $\cDAK$.
\end{theorem}
\begin{proof}
The original statement and proof of this result is given as \cite[Theorem 1.1 (3)]{EHKM17} over the field of complex numbers. A proof for the case where $\K=\R$ is given in Section \ref{sec:EHKMReal}. We comment that the original statement handles more general complex dimension free sets; however, this version is well suited to our needs.
\end{proof}

Our next theorem shows that every element of a compact free spectrahedron $\cDAK$ which is closed under complex conjugation dilates to the Arveson boundary of $\cDAK$.

\begin{thm}
\label{thm:AbsSpanBound}
Let $A$ be a $g$-tuple of self-adjoint matrices with entries in $\K$ and let $\cDAK$ be a compact free spectrahedron which is closed under complex conjugation. Let $X \in \cDAK (n)$ with
\[
\dim \fK^\K_{A,X}=\ell.
\]

\begin{enumerate}

\item 
\label{it:ComplexSpecBound}
There exists an integer $k \leq 2\ell+n \leq 2ng+n$ and $k$-dilation $Y$ of $X$ such that $Y$ is an Arveson extreme point of $\cDAC$. Thus, $X$ is a matrix convex combination of absolute extreme points of $\cDAC$ whose sum of sizes is equal to $n+k$.

\item
\label{it:RealSpecBound}
Suppose $X$ is a tuple of real symmetric matrices, then there exists an integer $k \leq \ell \leq ng$ and $k$-dilation $Y$ of $X$ such that $Y$ is an Arveson extreme point of $\cDAK$. Thus, $X$ is a matrix convex combination of absolute extreme points of $\cDAK$ whose sum of sizes is equal to $n+k$.
\end{enumerate}
As an immediate consequence, $\cD_A^\K$ is the matrix convex hull of its absolute extreme points.
\end{thm}
\begin{proof}

The proof that $X \in \cDAR$ dilates to an Arveson extreme point of $\cDAR$ is given in Section \ref{sec:AbsSpanBoundProof}. We prove that $X \in \cDAC$ dilates to an Arveson extreme point of $\cDAC$ in Section \ref{sec:ComplexSpectrahedra}.

We now prove that $\cD_A^\K$ is the matrix convex hull of its absolute extreme points. Let $X \in \cDAK$. The first part of Theorem \ref{thm:AbsSpanBound} shows that, in the complex setting, there is an Arveson extreme point $Y \in \cDAK (n+k)$ for some $k \leq 2\dim \fK^\K_{A,X}+n$ such that $X$ is a compression of $Y$.

The $g$-tuple $Y$ is unitarily equivalent to a direct sum of $m$ irreducible tuples $\{Y^i\}_{i=1}^m$ for some integer $m$. These too are Arveson, hence absolute, extreme points, see Theorem \ref{theorem:EHKMRealComplex}.  Since $X$ is a compression of $Y$, it follows that $X$ is a compression of $\oplus_{i=1}^m Y^i$. Equivalently, there is an isometry $V:\K^n \to \K^{n+k}$ such that 
$X= V^* (\oplus_{i=1}^m Y^i)V$.
 Decomposing
$V^*=\begin{pmatrix} V_1^* & \cdots & V_m^* \end{pmatrix}$
with respect to the block structure of $(\oplus_{i=1}^m Y^i)$ gives
\beq
\label{eq:XMatConvOfAbs}
X=\sum_{i=1}^m V_i^* Y^i V_i \quad \quad \sum_{i=1}^m V_i^* V_i=I_n \quad \quad \mathrm{with\ } Y^i \in \cDAK (n_i) \mathrm{\ and \ } \sum_{i=1}^m n_i =n+k.
\eeq
That is, $X$ is a matrix convex combination of the absolute extreme points $Y^1, \dots, Y^m$.

The proof when $X$ is a $g$-tuple of $n \times n$ real symmetric matrices is identical with $n+k$ replaced by $n+\tilde{k}$ where $\tilde{k} \leq \dim \fK_{A,X}^\K.$
\end{proof}

We comment that there are examples of a free spectrahedron $\cDAK$ and an irreducible tuple $X \in \cDAK$ and an Arveson dilation $Y$ of $X$ that has minimal size such that $Y$ is reducible. 

\subsection{Reader's guide}

Section \ref{sec:ProofsComps} introduces the notion of a maximal $1$-dilation of an element of a free spectrahedron. The main result of this section is Theorem \ref{theorem:1DiaReduceKerConDim} which implies that, in the real setting, Arveson dilations of a tuple $X \in \cDAR$ can be constructed by taking a sequence of maximal $1$-dilations of $X$. This result is then used to prove Theorem \ref{thm:AbsSpanBound} \eqref{it:RealSpecBound} when $\K=\R$. The section ends with Proposition \ref{prop:ArvDiAlg} which gives a numerical algorithm that can be used to construct Arveson dilations of elements of a real free spectrahedron.

Section \ref{sec:ComplexSpectrahedra} considers the case where $\K=\C$ and  completes the proof of Theorem \ref{thm:AbsSpanBound}. This is accomplished by showing that, when $\cDAC$ is closed under complex conjugation, the absolute extreme points of $\cDAR$ are absolute extreme points of $\cDAC$. We then show that every element of a complex free spectrahedron which is closed under complex conjugation is a compression of an element of the associated real free spectrahedron. An appeal to Theorem \ref{theorem:1DiaReduceKerConDim} completes the proof. In addition this section gives a classification of free spectrahedra which are closed under complex conjugation.

Section \ref{sec:Remarks} expands on the historical context of our main results. Section \ref{sec:FreeCaratheodory} describes a count on the number of parameters needed to express a tuple as a matrix convex combination of absolute extreme points which is given by Theorem \ref{thm:AbsSpanBound}. Section \ref{sec:GeneralMatConv} compares our results to results for general matrix convex sets, and Section \ref{sec:AltContext} discusses the original terminology and viewpoint of \cite{A69}, \cite{DM05}, and \cite{DK15}.

An appendix, Section \ref{sec:NCLDL}, contains a discussion of the NC LDL$^*$ calculation which appears in the proof of Theorem \ref{theorem:1DiaReduceKerConDim}. In addition, the appendix contains a proof of the real analogue of \cite[Theorem 1.1 (3)]{EHKM17}.

The authors thank Igor Klep and Scott McCullough for comments on the original version of this manuscript.

\section{Real free spectrahedra}
\label{sec:ProofsComps}
We first consider the case of Theorem \ref{thm:AbsSpanBound} where $X$ is an element of $\cDAR$ . We begin with a collection of lemmas and definitions which will play an important role in the proof of this case.

\subsection{The dilation subspace}
\label{sec:DiSubspace}

The subspace $\fK^\K_{A,X}$ is called the dilation subspace since, by considering the Schur complement, a tuple $\beta \in M_{n \times 1} (\K)^g$ is an element of $\fK^\K_{A,X}$ if and only if there is a real number $c > 0$ and a tuple $\gamma \in \R^g$ such that
\beq
\label{eq:DiSubConnection}
Y=\begin{pmatrix}
X & c\beta \\
c\beta^* & \gamma
\end{pmatrix} \in \cDAK.
\eeq

The following lemma explains the relationship between the dilation subspace $\fK^\K_{A,X}$ and dilations of the tuple $X \in \cDAK$ in greater detail. 

\begin{lem}
\label{lem:DiSubProperties}
Let $\cDAK$ be a free spectrahedron and let $X \in \cDAK (n)$. 
\begin{enumerate}
\item
\label{it:DiBetaIsInDiSub} If $\beta \in M_{n \times 1} (\K)^g$ and 
\[
Y= \begin{pmatrix}
X & \beta \\
\beta^* & \gamma
\end{pmatrix} \in \cDAK (n+1)
\]
is a $1$-dilation of $X$, then $\beta \in \fK^\K_{A,X}$. 
\item
\label{it:DiSubDilatesX}
Let $\beta \in M_{n \times 1} (\K)^g$. Then $\beta \in \fK^\K_{A,X}$ if and only if there is a tuple $\gamma \in \cDAK (1)$ real number $c_\gamma>0$ such that
\[
\begin{pmatrix}
X & c_\gamma \beta \\
c_\gamma \beta^* & \gamma
\end{pmatrix} \in \cDAK (n+1).
\]
In particular, one may take $\gamma=0 \in \K^g$.
\item 
\label{it:XArvIFFDiSubIsZero}
$X$ is an Arveson extreme point of $\cDAK$ if and only if $\dim \fK^\K_{A,X}=0$.
\end{enumerate}
\end{lem}

\begin{proof}
Item \eqref{it:DiBetaIsInDiSub} follows from considering the Schur complement of $L_A (Y)$ for a dilation 
\[
Y= \begin{pmatrix}
X & \beta \\
\beta^* & \gamma
\end{pmatrix} \in \cDAK (n+1)
\]
of $X$. Indeed, multiplying $L_A (X)$ by permutation matrices, sometimes called canonical shuffles, see \cite[Chapter 8]{P02}, shows 
\beq
\label{eq:2x2CanShuf}
L_A (Y) \succeq 0 \quad \mathrm{ if \ and \ only \ if } \quad \begin{pmatrix} L_A (X) & \Lambda_A (\beta) \\ \Lambda_A (\beta^*) & L_A (\gamma) \end{pmatrix} \succeq 0.
\eeq
Taking the appropriate Schur complement then implies that
\beq
\label{eq:2x2SchurComp}
L_A (Y) \succeq  0 \mathrm{\ if\ and\ only\ if\ } L_A (\gamma) \succeq 0 \mathrm{\ and\ } L_A (X)-\Lambda_A (\beta) L_A (\gamma)^\dagger \Lambda_A (\beta^*) \succeq 0
\eeq
where $\dagger$ denotes the Moore-Penrose pseudoinverse. Item \eqref{it:DiBetaIsInDiSub} is an immediate consequence of equation \eqref{eq:2x2SchurComp}.  See \cite[Corollary 2.3]{EHKM17} for a related argument.

We now prove item \eqref{it:DiSubDilatesX}. Note that $L_A (0)=I$, so similar to before using the Schur complement shows
\[
Y_0=\begin{pmatrix}
X & c\beta \\
c\beta^* & 0
\end{pmatrix} \in \cDAK (n+1)
\]
if and only if
\beq
\label{eq:2x2SchurCompOnZero}
 L_A (X)-c^2 \Lambda_A (\beta) \Lambda_A (\beta^*) \succeq 0.
\eeq
If $\beta \in \fK^\K_{A,X}$ then $\ker L_A (X) \subset \ker \Lambda_A (\beta) \Lambda_A (\beta^*)$, so picking $c$ small enough so that $\|c^2 \Lambda_A (\beta) \Lambda_A (\beta^*)\|_2$ is less than the smallest nonzero eigenvalue of $L_A (X)$ guarantees that inequality \eqref{eq:2x2SchurCompOnZero} holds, hence $Y_0 \in \cDAK (n+1)$. The reverse direction is a consequence of item \eqref{it:DiBetaIsInDiSub}.

Item \eqref{it:XArvIFFDiSubIsZero} follows from items \eqref{it:DiBetaIsInDiSub} and \eqref{it:DiSubDilatesX}.
\end{proof}

\begin{remark}
The choice of $\gamma=0 \in \K^g$ in Lemma \ref{lem:DiSubProperties} \eqref{it:DiSubDilatesX} helps simplify the NC LDL$^*$ calculations used in the forthcoming proof of Theorem \ref{theorem:1DiaReduceKerConDim}.
\end{remark}

\subsection{Maximal 1-dilations}

An important aspect of the proof of our main result is constructing dilations which satisfy a notion of maximality. Given a matrix convex set $K$ and a tuple $X \in K (n)$, say the dilation 
\[
Y=\begin{pmatrix}
X & c\hbeta \\
c\hbeta^* & \hgamma
\end{pmatrix} \in K(n+1)
\]
is a \df{maximal $1$-dilation} of $X$ if $Y$ is a $1$-dilation of $X$ and $\hbeta$ is nonzero and the real number $c$ and tuple $\hgamma \in \R^g$ are solutions to the sequence of maximization problems 
\[
\begin{array}{rllcl}  c:=&\underset{\alpha \in \R, \gamma \in \R^g}{\mathrm{Maximizer}} \ \ \ \ \alpha \\
\mathrm{s.t.} & \begin{pmatrix}

X & \alpha \hbeta \\
\alpha \hbeta^* & \gamma
\end{pmatrix} \in K(n+1) \\
\\
\mathrm{and } \quad \hgamma:=&\underset{ \gamma \in \R^g}{\mathrm{A \ Local \ Maximizer}} \ \   \|\gamma\| \\
\mathrm{s.t.}  &\begin{pmatrix}

X & c \hbeta \\
c \hbeta^* & \gamma
\end{pmatrix} \in K (n+1)
\end{array}
\]
where $\| \cdot \|$ denotes the usual norm on $\R^g$. We note that maximal $1$-dilations can be computed numerically, see Proposition \ref{prop:ArvDiAlg}. We emphasize that $\hgamma$ produced by the second optimization need only be any local maximizer, and global maximality is not required.

\begin{remark}
\label{remark:MaxDiExist}
\rm If $K$ is a compact matrix convex set and $X \in K$ is not an Arveson extreme point of $K$, then a routine compactness argument shows the existence of nontrivial maximal $1$-dilations of $X$.  \qed
\end{remark}

Other notions of maximal dilations (in the infinite dimensional setting) are discussed in \cite{DM05}, \cite[Section 2]{A08} and \cite[Section 1]{DK15}.

\subsection{Maximal dilations reduce the dimension of the dilation subspace}

Let $A \in \smdrg$, let $\cDAR$ be a compact real free spectrahedron, and let $X \in \cDAR$. The following theorem shows that maximal $1$-dilations of $X$ reduce the dimension of the dilation subspace. 

\begin{theorem}
\label{theorem:1DiaReduceKerConDim}
Let $A \in \smdrg$ be a $g$-tuple of self-adjoint matrices over $\R$ such that $\cDAR$ is a compact real free spectrahedron and let $X \in \cDAR (n)$. Assume $X$ is not an Arveson extreme point of $\cDAR$. Then there exists a nontrivial maximal 1-dilation $\hY \in \cDAR (n+1)$ of $X$. Furthermore, any such $\hY$ satisfies 
\[
\dim \fK^\R_{A, \hY} < \dim \fK^\R_{A,X}.
\]
\end{theorem}

\begin{proof}

Let $\hY$ be a maximal $1$-dilation of $X$. Equivalently, choose the dilation $\hY$ (choose $\hbeta$ and $\hgamma)$ such that
\[
\hY=\begin{pmatrix}
X & \hbeta \\
\hbeta^* & \hgamma
\end{pmatrix} \mathrm{ \ is \ in \ } \cDAR (n+1),
\]
and if
\[
\tilde{Y}_c= \begin{pmatrix} X & c \hbeta \\
c\hbeta^* & \gamma
\end{pmatrix} \mathrm{ \ is \ in \ } \cDAR (n+1)
\]
for a tuple $\gamma \in \R^g$ and a real number $c \in \R$, then $c \leq 1$.\footnote{ If $\tilde{Y}_c$ is an element of $\cDAR (n+1)$ then so is $\tilde{Y}_{-c}$. For this reason, it is equivalent to require $|c| \leq 1$. } Furthermore, if $c=1$ and $\tilde{Y} \in \cDAR (n+1)$, then there exists an $\epsilon>0$ such that $\|\hgamma-\gamma\|<\epsilon$ implies $\|\gamma\| \leq \|\hgamma\|$. As mentioned in Remark \ref{remark:MaxDiExist}, the existence of such a $\hY$ follows from the assumptions that $X$ is not an Arveson extreme point of $\cDAR$ and that $\cDAR$ is level-wise compact. 

We will show that
\[
\dim \fK^\R_{A,\hY} < \dim \fK^\R_{A,X}.
\]
First consider the subspace
\[
\fE_{A,\hY}:=\{ \eta \in M_{n \times 1} (\R)^g \  | \mathrm{\ there \ exists \ a \ } \sigma\in \R^g \mathrm{\ so \ that \ } \ker L_A (\hY) \subseteq \ker \Lambda_A \begin{pmatrix} \eta^* & \sigma \end{pmatrix} \}.
\]
\index{$\fE_{A,\hY}$}
In other words $\fE_{A,\hY}$ is the projection $\iota$ of $\fK^\R_{A,\hY}$ defined by
\[
\fE_{A,\hY}:=\iota (\fK^\R_{A,X}) \mathrm{\ where \ } \iota \begin{pmatrix} \eta \\
\sigma 
\end{pmatrix}=\eta
\]
for $\eta \in M_{n \times 1} (\R)^g$ and $\sigma \in \R^g$. We will show $\dim \fE_{A,\hY} < \dim \fK^\R_{A,X}$.

If $\eta \in \fE_{A,\hY}$, then there exists a tuple $\tilde{\sigma}\in \R^g$ such that 
\[
\begin{pmatrix} \eta^* & \tilde{\sigma} \end{pmatrix} \in \fK^\R_{A,\hY}.
\]
From Lemma \ref{lem:DiSubProperties} \eqref{it:DiSubDilatesX}, it follows that there is a real number $c>0$ so that setting $\sigma=c \tilde{\sigma}$ gives
\[
\begin{pmatrix}
X & \hbeta & c\eta \\
\hbeta^* & \hgamma & \sigma \\
c\eta^* &  \sigma ^* & 0 
\end{pmatrix}
\in \cDAR.
\]
Since $\cDAR$ is matrix convex it follows that
\[
\begin{pmatrix}
1 & 0 & 0 \\
0 & 0 & 1
\end{pmatrix}
\begin{pmatrix}
X & \hbeta & c\eta \\
\hbeta^* & \hgamma & \sigma \\
c\eta^* &  \sigma ^* & 0 
\end{pmatrix}
\begin{pmatrix}
1 & 0  \\
0 & 0 \\
0 & 1
\end{pmatrix}=
\begin{pmatrix} 
X & c \eta \\
c \eta^* & 0 
\end{pmatrix} \in \cDAR,
\]
so Lemma \ref{lem:DiSubProperties} \eqref{it:DiBetaIsInDiSub} shows $\eta \in \fK^\R_{A,X}$. In particular this shows
\beq
\label{eq:diKerConSubset}
\fE_{A,\hY} \subset \fK^\R_{A,X}.
\eeq

Now, assume towards a contradiction that 
\[
\dim \fE_{A,\hY} =\dim \fK^\R_{A,X}.
\]
Using equation \eqref{eq:diKerConSubset} this implies that
\[
\fE_{A,\hY} = \fK^\R_{A,X}.
\]
In particular we have $\hbeta \in \fE_{A,\hY}$. It follows that there is a real number $c \neq 0$ and a tuple $\sigma \in \R^g$ so that
\beq
\label{eq:badBetaDi}
L_A\begin{pmatrix}
X & \hat{\beta} & c \hat{\beta} \\
\hat{\beta}^* & \hat{\gamma} & \sigma \\
c\hat{\beta}^* & \sigma & 0
\end{pmatrix}\succeq 0.
\eeq

Using the NC LDL$^*$-decomposition (up to canonical shuffles) shows that inequality \eqref{eq:badBetaDi} holds if and only if $L_A (X) \succeq 0$ and the Schur complements
\beq
\label{eq:3LDLMiddle}
I_d-c^2 Q \succeq 0
\eeq
and
\beq
\label{eq:3LDL}
L_A (\hgamma)-Q-\left(\Lambda_A(\sigma)-cQ\right)^*\left(I_d-c^2 Q\right)^\dagger \left(\Lambda_A(\sigma)-cQ\right) \succeq 0
\eeq
where 
\beq
\label{eq:Qdef}
Q:= \Lambda_A (\hbeta^*) L_A (X)^\dagger \Lambda_A (\hbeta).
\eeq
It follows that
\beq
\label{eq:gammaSchur}
L_A (\hat{\gamma})-Q \succeq 0
\eeq
and
\beq
\label{eq:3diaKerContain}
\ker  [L_A (\hat{\gamma})-Q] \subseteq \ker [\Lambda_A(\sigma)-c Q].
\eeq

Inequalities \eqref{eq:gammaSchur} and \eqref{eq:3diaKerContain} imply that there exists a real number $\tilde{\alpha} >0$ such that $0 < \alpha \leq \tilde{\alpha}$ implies
\[
L_A (\hat{\gamma})-Q \  \pm \ \alpha \left(\Lambda_A(\sigma)-c Q\right) \succeq 0.
\]
It follows from this that
\beq
\label{eq:betterBetaSchur}
\begin{array}{rllcl}
&L_A (\hat{\gamma} \pm \alpha \sigma)- (1 \pm c\alpha)Q\\
=& L_A (\hat{\gamma} \pm \alpha \sigma)-\left(\Lambda_A (\sqrt{1 \pm c\alpha}\hat{\beta}^*)L_A(X)^\dagger\Lambda_A(\sqrt{1 \pm c\alpha}\hat{\beta})\right)& \succeq & 0.
\end{array}
\eeq
Since $L_A(X) \succeq 0$, equation \eqref{eq:betterBetaSchur} implies
\beq
\label{eq:betterBetaDi}
L_A \begin{pmatrix}
X & \sqrt{1 \pm c\alpha}\hat{\beta} \\
\sqrt{1 \pm c\alpha}\hat{\beta}^* & \hat{\gamma}\pm \alpha \sigma
\end{pmatrix}\succeq 0.
\eeq
Therefore, from our choice of $\hY$, hence of $\hbeta$, we must have 
\[
\sqrt{1 \pm c\alpha}  \leq 1.
\]
It follows that $c\alpha=0$. However, we have assumed $\alpha>0$ and $c \neq 0$, so this is a contradiction. We conclude
\[
\dim \fE_{A,\hY} < \dim \fK^\R_{A,X}.
\]

Now seeking a contradiction assume $\dim \fK^\R_{A,\hY}= \dim \fK^\R_{A,X}$. Since $\dim \fE_{A,\hY} < \dim \fK^\R_{A,X}$, there must exist tuples $\eta \in M_{n \times 1} (\R)^g$ and $\sigma^1, \sigma^2 \in \R^g$ such that $\sigma^1 \neq \sigma^2$ and so
\[
\begin{pmatrix}
\eta \\
\sigma^1
\end{pmatrix},
\begin{pmatrix}
\eta \\
\sigma^2
\end{pmatrix} \in \fK^\R_{A,\hY}.
\]
It follows that 
\beq
\label{eq:hsigmaKerCon}
\begin{pmatrix}
0 \\
\sigma^1-\sigma^2
\end{pmatrix} \in \fK^\R_{A,\hY}.
\eeq
Set $\hsigma=\sigma^1-\sigma^2 \neq 0 \in \R^g$. As before, equation \eqref{eq:hsigmaKerCon} with Lemma \ref{lem:DiSubProperties} \eqref{it:DiSubDilatesX} implies that there is a real number $c \neq 0 \in \R$ so that
\beq
\label{eq:BadSigDi}
L_A
\begin{pmatrix}
X & \hbeta & 0 \\
\hbeta^* & \hgamma & c \hsigma \\
0 & c \hsigma & 0 
\end{pmatrix} 
\succeq 0.
\eeq
Considering the NC LDL$^*$ decomposition shows that equation \eqref{eq:BadSigDi} holds if and only if
\beq
\label{eq:3LDLSig}
 L_A (X)\succeq 0 \quad \ \mathrm{and} \ \quad  L_A (\hat{\gamma})-Q-c^2 \Lambda_A (\hsigma)\Lambda_A(\hsigma) \succeq 0,
\eeq
where $Q=\Lambda_A(\hat{\beta^*})L_A (X)^\dagger \Lambda_A(\hat{\beta})$ as before. It follows from this that
\beq
\label{eq:3diaKerContainSig}
\ker [ L_A (\hat{\gamma})-Q] \subseteq \ker \Lambda_A(\hsigma) \quad \mathrm{and} \quad L_A (\hat{\gamma})-Q \succeq 0.
\eeq
This implies that there is a real number $\tilde{\alpha}>0$ so that, for all $\alpha \in \R$ satisfying $0 < \alpha \leq \tilde{\alpha}$, we have
\[
 L_A (\hat{\gamma})-Q\pm \Lambda_A (\alpha \hsigma) = L_A (\hat{\gamma} \pm \alpha \hsigma)-Q\succeq 0.
\]
Since this is the appropriate Schur complement and since $L_A (X) \succeq 0$ it follows that 
\beq
\label{eq:hsigBadNorm}
L_A \begin{pmatrix}
X & \hbeta \\
\hbeta^* & \hgamma \pm \alpha \hsigma 
\end{pmatrix} \succeq 0
\eeq
whenever $0< \alpha \leq \tilde{\alpha}$. Therefore, the local maximality of $\hgamma$ implies
\[
\| \hgamma+\alpha \hsigma\| \leq \|\hgamma\| \mathrm{\ and \ } \| \hgamma-\alpha \hsigma\| \leq \|\hgamma\|
\]
for sufficiently small $\alpha \in (0, \tilde{\alpha}]$, a contradiction to the assumptions that $\alpha \neq 0$ and $\hsigma \neq 0$. We conclude that $\dim \fK^\R_{A,\hY} < \dim \fK^\R_{A, X}$ as asserted by Theorem \ref{theorem:1DiaReduceKerConDim}. 
\end{proof}

\subsubsection{Proof of Theorem \ref{thm:AbsSpanBound} for real free spectrahedra}
\label{sec:AbsSpanBoundProof}
We are now in position to prove Theorem \ref{thm:AbsSpanBound} in the case where $\cDAR$ is a compact real free spectrahedron.

\noindent{\it Proof of Theorem \ref{thm:AbsSpanBound} when $\K=\R$.} Given a tuple $X \in \cDAR$ with $\dim \fK^\R_{A,X}=\ell$, the existence of a $k$-dilation $Y$ of $X$ such that $Y \in \arv \cDAR$ for some $k \leq \ell$ is an immediate consequence of Theorem \ref{theorem:1DiaReduceKerConDim} and Lemma \ref{lem:DiSubProperties} \eqref{it:XArvIFFDiSubIsZero}.

The fact that $\cDAR$ is the matrix convex hull of its Arveson extreme points, hence of its absolute extreme points, is proved immediately after the statement of Theorem \ref{thm:AbsSpanBound}. $\hfill\qedsymbol$

\subsection{Numerical computation}
\label{sec:Computation}

Given a compact real free spectrahedron $\cDAR$, the following algorithm dilates a tuple $X \in \cDAR$ to an Arveson extreme point $Y \in \cDAR$ in $\dim \fK^\R_{A,X}$ steps or less.

\begin{prop}
\label{prop:ArvDiAlg}
Let $A \in \smdrg$ be a $g$-tuple of self-adjoint matrices over $\R$ such that $\cDAR$ is a compact real free spectrahedron. Given a tuple $X \in \cDAR (n)$, set $Y^0=X$. For integers $k=0,1,2\dots$ and while $\dim \fK^\R_{A,Y^k} >0$ define 
\[
Y^{k+1}:=\begin{pmatrix}
Y^k & c_k \hbeta^k \\
c_k (\hbeta^k)^* & \hgamma^k
\end{pmatrix}
\]
where $\hbeta^k$ is any nonzero element of $\fK^\R_{A,Y^k}$ and 
\[
\begin{array}{rllcl}  c_k:=&\underset{c \in \R, \gamma \in \R^g}{\mathrm{Maximizer}} \ \ \ \ c \\
\mathrm{s.t.} &L_A \begin{pmatrix}

Y^k & c\hbeta^k \\
c(\hbeta^k)^* & \gamma
\end{pmatrix}\succeq 0 ,\\
\\
\mathrm{and } \quad \hgamma^k:=&\underset{\gamma \in \R^g}{\mathrm{A \ Local \ Maximizer}} \ \ \|\gamma\| \\
\mathrm{s.t.}  &L_A \begin{pmatrix}

Y^k & c_k \hbeta^k \\
c_k (\hbeta^k)^* & \gamma
\end{pmatrix}\succeq 0.
\end{array}
\]
Then $\dim \fK^\R_{A,Y^\ell}=0$ for some integer $\ell \leq \dim \fK^\R_{A,X} \leq ng$ and $Y^\ell$ is an Arveson $\ell$-dilation of $X$.
\end{prop}
\begin{proof}
This follows from the proof of Theorem \ref{theorem:1DiaReduceKerConDim}.
\end{proof}

The optimization over $c$ in Proposition \ref{prop:ArvDiAlg} is a semidefinite program, while the optimization over $\gamma$ is a local maximization of a convex quadratic over a spectrahedron.

\section{Complex free spectrahedra}
\label{sec:ComplexSpectrahedra}

This section will prove that every element of a compact complex free spectrahedron which is closed under complex conjugation is the matrix convex hull of its absolute extreme points. We begin with a lemma which shows that the set of real elements in the absolute boundary of a complex free spectrahedron $\cDAC$ which is closed under complex conjugation is exactly equal to the absolute boundary of $\cDAR$.

\begin{lem}
\label{lem:RealArvIsComplexArv}
Let $A$ be a $g$-tuple of $d \times d$ real symmetric matrices and let $X \in \cDAC$ be a $g$-tuple of $n \times n$ real symmetric matrices. Then $X$ is an Arveson extreme point of $\cDAC$ if and only if $X$ is an Arveson extreme point of $\cDAR$. 

\end{lem}

\begin{proof}
It is straightforward to show that $X$ is an Arveson extreme point of $\cDAR$ if $X$ is an Arveson extreme point of $\cDAC$. To prove the converse, assume $X$ is an Arveson extreme point of $\cDAR$ and let $\beta \in M_{n \times 1} (\C)^g$ be a tuple such that 
\[
\begin{pmatrix}
X & \beta \\
\beta^* & \gamma
\end{pmatrix} \in \cDAC.
\]

By assumption $A$ is a tuple of real symmetric matrices so $\cDAC$ is closed under complex conjugation. It follows that
\[
\overline{
\begin{pmatrix}
X & \beta \\
\beta^* & \gamma
\end{pmatrix}}
=
\begin{pmatrix}
X & \overline{\beta} \\
\overline{\beta}^* & \gamma
\end{pmatrix}
\in \cDAC.
\]
Since $\cDAC$ is convex we conclude that 
\[
\begin{pmatrix}
X & \mathrm{Re} (\beta) \\
\mathrm{Re} (\beta)^* & \gamma 
\end{pmatrix}=\frac{1}{2}
\left(\begin{pmatrix}
X & \beta \\
\beta^* & \gamma
\end{pmatrix}+
\begin{pmatrix}
X & \overline{\beta} \\
\overline{\beta}^* & \gamma
\end{pmatrix}\right) \in \cDAC.
\]
This matrix has real entries so it is an element of $\cDAR$. However, $X$ was assumed to be an Arveson extreme point of $\cDAR$ so we must have $\mathrm{Re} (\beta)=0.$

Now, $\cDAC$ is closed under unitary conjugation so we know 
\[
\begin{pmatrix}
X & i\beta \\
(i\beta)^* & \gamma
\end{pmatrix}
=\begin{pmatrix}
1 & 0 \\
0 & -i
\end{pmatrix}
\begin{pmatrix}
X & \beta \\
\beta^* & \gamma
\end{pmatrix}
\begin{pmatrix}
1 & 0 \\
0 & i
\end{pmatrix} \in \cDAC.
\]
However, this matrix is in $\cDAR$ since $\mathrm{Re} (\beta)=0$ from which it follows that $\mathrm{Im} (i \beta)=0$.
We have assumed that $X$ is an Arveson extreme point of $\cDAR$, so $i\beta=0$, hence $\beta=0$. We conclude that $X$ is an Arveson extreme point of $\cDAC$, as claimed. 
\end{proof}

Our next lemma gives a list of equalities for the dilation subspace which will be used in proving the bound on the dimension of the absolute extreme points appearing in Theorem \ref{thm:AbsSpanBound}.

\begin{lem}
\label{lem:DiSubEqsAndDims}
Let $\cDAK$ be a real or complex free spectrahedron. The following equalities hold for the dilation subspace:

\begin{enumerate}
\item
\label{it:DiSubDirectSum} Let $X  \in \cDAK (n_1)$ and $Z \in \cDAK (n_2)$. Then 
\[
\fK_{A,X \oplus Z}^\K=\left\{\begin{pmatrix} \beta^* & \sigma^* \end{pmatrix}^* \in M_{(n_1+n_2) \times 1} (\K)^g \big|\ \beta \in \fK_{A,X}^\K \ \mathrm{ and \ } \sigma \in \fK_{A,Z}^\K  \right\}.
\]
Additionally, 
\[
\fK_{A,X \oplus Z}^\K= \dim \fK_{A,X}^\K+ \dim \fK_{A,Z}^\K.
\]

\item
\label{it:DiSubUnitary}
 Let $X \in \cDAK(n)$ and let $U \in M_n (\K)$ be a unitary. Then
\[
\fK_{A,X}^\K= U^* \fK^\K_{A,U^* X U} \quad \mathrm{ and } \quad
\dim \fK_{A,X}^\K= \dim \fK^K_{A,U^* X U}\ .
\]

\item
\label{it:DiSubComplexConjugation} 
Assume $\cDAK$ is closed under complex conjugation. Then 
\[
\fK_{A,X}^\K=\overline{\fK^\K_{A,\overline{X}}} \quad \mathrm{ and } \quad \dim \fK_{A,X}^\K= \dim \fK^\K_{A,\overline{X}}\ .
\]
\end{enumerate}
\end{lem}

\begin{proof}
The proof of item \eqref{it:DiSubDirectSum} is immediate from the fact that $ \ker L_A (X \oplus Z) \subset \ker \Lambda_A \begin{pmatrix}
\beta^* & \sigma^*
\end{pmatrix}
$ 
if and only if $\ker L_A (X) \subset \ker \Lambda_A (\beta^*)$ and $\ker L_A (Z) \subset \ker \Lambda_A (\sigma^*)$.

To prove item \eqref{it:DiSubUnitary} let $U \in M_n (\K)$ be a unitary and observe that 
\[
\begin{pmatrix} X & \beta \\
\beta^* & \gamma 
\end{pmatrix} \in \cDAK
\quad \iff \quad
\begin{pmatrix}
U^* X U & U^* \beta \\
\beta^* U & \gamma 
\end{pmatrix}=
\begin{pmatrix}
U^* & 0 \\
0 & 1 
\end{pmatrix}
\begin{pmatrix} X & \beta \\
\beta^* & \gamma 
\end{pmatrix}
\begin{pmatrix}
U & 0 \\
0 & 1 
\end{pmatrix} \in \cDAK. 
\]

To prove item \eqref{it:DiSubComplexConjugation}: assume $\cDAK$ is closed under complex conjugation. Then 
\[
\begin{pmatrix} X & \beta \\
\beta^* & \gamma 
\end{pmatrix} \in \cDAK
\quad \iff \quad
\begin{pmatrix} \overline{X} & \overline{\beta} \\
\overline{\beta^*} & \overline{\gamma}
\end{pmatrix} \in \cDAK.
\] \end{proof}

We now give a classification of free spectrahedra which are closed under complex conjugation.

\begin{lem}
\label{lem:ComplexConjAReal}
Let $A$ be a $g$-tuple of $d \times d$ complex self-adjoint matrices. Then  the complex free spectrahedron $\cDAC$ is closed under complex conjugation if and only if there is a $g$-tuple $B$ of real symmetric matrices of size less than or equal to $2d \times 2d$ such that $\cDAC= \cD_B^{\mathbb{C}}$.
\end{lem}

\begin{proof}
We first prove the forwards direction. Let $X$ be a $g$-tuple of complex self-adjoint matrices. Since $\cDAC$ is closed under complex conjugation we know that $X \in \cDAC$ if and only if
\beq
\label{eq:ConjToA}
L_A (X) \succeq 0 \quad \mathrm{and} \quad L_{\overline{A}} (X) \succeq 0.
\eeq
Thus $X \in \cDAC$ if and only if $L_{A \oplus \overline{A}} (X) \succeq 0.$

Write $A=S+iT$ where $S$ is a tuple of $n \times n$ real symmetric matrices and $T$ is a tuple of $n \times n$ real skew symmetric matrices. Then $A \oplus \overline{A}$ is unitarily equivalent to the $g$-tuple of real  symmetric matrices $B$ defined by
\beq
\label{eq:ComplexToReal}
B:=\begin{pmatrix}
S & -T \\
T & S
\end{pmatrix}=U^*
\begin{pmatrix}
S+iT & 0 \\
0 & S-iT
\end{pmatrix}
U
\eeq
where $U \in M_{2n} (\C)$ is the unitary
\[
U= \frac{\sqrt{2}}{2}
\begin{pmatrix}
I_n & i I_n \\
i I_n & I_n 
\end{pmatrix}. 
\] We conclude that $X \in \cDAC$ if and only if
\[
L_{B} (X) \succeq 0.
\]
It follows that $\cDAC=\cD_B^{\mathbb{C}}$.

The converse is straightforward. \end{proof}

We are now in position to complete the proof of the Theorem \ref{thm:AbsSpanBound}.

\noindent{\it Proof of Theorem \ref{thm:AbsSpanBound}.} 
Let $\cDAC$ be a compact complex free spectrahedron which is closed under complex conjugation and let $X \in \cDAC (n)$. In light of Lemma \ref{lem:ComplexConjAReal}, we may without loss of generality assume that $A$ is a $g$-tuple of real symmetric matrices. Set $\ell= \dim \fK^\C_{A,X}.$ If $X$ is an element of $\cDAR$, that is, if $X$ is a tuple of real symmetric matrices, then the proof that $X$ dilates to an Arveson extreme point $Y \in \cDAC(n+k)$ for some integer $k \leq \ell$ is immediate from Theorem \ref{theorem:1DiaReduceKerConDim} with Lemma \ref{lem:RealArvIsComplexArv}. 

 To handle the general case where $\mathrm{Im}(X) \neq 0$, write $X=S+iT$ where $S$ is a $g$-tuple of $n \times n$ real symmetric matrices and $T$ is a $g$-tuple of $n \times n$ real skew symmetric matrices. By assumption $\cDAC$ is closed under complex conjugation so we know $S-iT \in \cDAC$. As shown in equation \eqref{eq:ComplexToReal}, the tuple $(S+iT) \oplus (S-iT)$ is unitarily equivalent to the tuple $Z \in \cDAC (2n)$ defined by
\[
Z:=\begin{pmatrix}
S & -T \\
T & S 
\end{pmatrix}.
\]
It follows that $X$ is a compression of $Z$. 

Observe that $Z$ is a tuple of $2n \times 2n$ real symmetric matrices so $Z \in \cDAC$ implies $Z \in \cDAR$. Furthermore, an application of Lemma \ref{lem:DiSubEqsAndDims} shows that $\dim \fK^\C_{A,Z}=2 \ell$, hence $\dim \fK^\R_{A,Z} \leq 2 \ell$. Theorem \ref{theorem:1DiaReduceKerConDim} shows that $Z$ dilates to an Arveson extreme point $\tilde{Z} \in \cDAR (2n+k)$ for some integer $k \leq 2\ell \leq 2ng$ and Lemma \ref{lem:RealArvIsComplexArv} implies that $\tilde{Z}$ is an Arveson extreme point of $\cDAC$. It follows that $X$ is a compression of the Arveson extreme point $\tilde{Z}.$

As in the real case, the proof that $\cDAC$ is the matrix convex hull of its absolute extreme points is given immediately after the statement of Theorem \ref{thm:AbsSpanBound}.$\hfill\qedsymbol$

\section{Remarks}
\label{sec:Remarks}

This section contains remarks which expand on the historical context of our results. Section \ref{sec:FreeCaratheodory} discusses the number of parameters needed to express a tuple as a matrix convex combination of absolute extreme points, while Section \ref{sec:GeneralMatConv} explores the relationship between the absolute extreme points of free spectrahedra and of general matrix convex sets. Section \ref{sec:AltContext} discusses infinite dimensional operator convex sets in Arveson's original context.

\subsection{Parameter counts for (matrix) convex combinations of extreme points}
\label{sec:FreeCaratheodory}
The classical Caratheodory Theorem gives an upper bound on how many terms are required to represent an element of a convex set as a convex combination of its extreme points. Theorem \ref{thm:AbsSpanBound} is the analog of this for a free convex set. In addition to giving a bound on the number of absolute extreme points needed to express an arbitrary tuple $X \in \cDAK(n)$, Theorem \ref{thm:AbsSpanBound} gives a bound on the number of parameters needed to express the absolute extreme points appearing in the matrix convex combination for $X$.

Given a compact free spectrahedron $\cDAK$, the classical Caratheodory Theorem states that a tuple $X \in \cDAK (n) \subset \smnkg$ can be written as a convex combination of 
$\dim \smnkg+1$ classical extreme points of $\cDAK (n)$, each an element of $\smnkg$. The maximum number of parameters in the extreme points required by this classical representation is
\[
(\dim \smnkg+1)(\dim \smnkg)=(n(n+1)g/2+1)(n(n+1)g/2)=O(n^4g^2).
\]

In contrast, Theorem \ref{thm:AbsSpanBound} shows that $X \in \cDAK(n)$ can be written as a matrix convex combination of a single Arveson extreme point $Y \in \cDAK (n+k)$ for some integer  $k \leq 2ng+n$. The maximum parameter count on the Arveson extreme point required in this dimension free representation is
\[
\dim \smspecialtwokg= 2(n+ng)(n+ng+1)g=O(n^2g^3).
\]

This suggests that matrix convex combinations are advantageous over classical convex combinations in terms of the number of parameters needed to store the representation of a tuple as a (matrix) convex combination of extreme points  when $n$ is large but that they are disadvantageous if $g$ is large.

\subsection{Absolute extreme points of general matrix convex sets}
\label{sec:GeneralMatConv}

Let $K \subset \smkg$ be a  compact matrix convex set. It is well known that there is a Hilbert space $\cH$ and a self-adjoint operator $\mathcal{A} \in B(\cH)$ such that $K= \cD^\K_{\cA}$, i.e., 
\[
K=\{X \in \smkg \ | \ L_{\cA} (X) \succeq 0 \},
\]
where $L_\cA (X)$ is defined as in the introduction \cite{EW97}. 

While Theorem \ref{thm:AbsSpanBound} shows every compact real free spectrahedron $\cDAR$ is spanned by its absolute extreme points, \cite[Theorem 1.2]{E17} shows the existence of a compact real matrix convex set $\cD^\R_{\mathcal{A}}$ which has no finite dimensional absolute extreme points. 

The critical failure of our proof for a general matrix convex set $\cD^\R_{\cA}$ occurs at equation \eqref{eq:3diaKerContain} in Theorem \ref{theorem:1DiaReduceKerConDim}. In Theorem \ref{theorem:1DiaReduceKerConDim} the tuple $A$ is finite dimensional, while $\cA$ being discussed here in Section \ref{sec:GeneralMatConv} is a tuple of operators acting on $\cH$ which may be infinite dimensional. Thus, the kernel containment
\[
\ker [L_{\cA} (\hgamma)-Q] \subset
\ker [\Lambda_\cA (\sigma)-cQ]
\]
along with
\[
L_{\cA} (\hgamma)-Q \succeq 0
\]
does not imply the existence of a real number $\alpha>0$ such that
\[
L_{\cA} (\hgamma)-Q \pm \alpha(\Lambda_\cA (\sigma)-cQ) \succeq 0.
\]
Here $Q=\Lambda_\cA (\hbeta^*) L_\cA (X)^\dagger \Lambda_\cA (\hbeta)$ similar to before.

A concrete example of this failure follows. Let $\cH=\ell^2 (\bbN)$, let $M=\mathrm{diag} (1/n^2)\in B(\cH),$ and let $N=\mathrm{diag} (1/n) \in B(\cH).$ Then $M \succeq 0$ and $\{0\}=\ker M \subseteq \ker N$, however $M-\alpha N \not\succeq 0$ for any real number $\alpha>0$.

\subsection{Alternative contexts}
\label{sec:AltContext}

Much of the literature such as \cite{A69}, \cite{DM05}, and \cite{DK15} referred to in the introduction takes a different viewpoint than the one here. We now briefly describe the correspondence.

Operator convex sets are in one to one correspondence with the set of completely positive maps on an operator system \cite{WW99}, an area which has received great interest over the last several decades. Under this correspondence, an absolute extreme point of an operator convex set becomes a boundary representation of an operator system \cite{KLS14}.

Arveson's original question was phrased in the setting of completely positive maps on an operator system. In this language, Arveson conjectured that every operator system has sufficiently many boundary representations to ``completely norm it". Additionally, Arveson conjectured that these boundary representations generate the $C^*$-envelope. Roughly speaking, the $C^*$-envelope of an operator system is the ``smallest"  $C^*$-algebra containing that operator system \cite{P02}. In this language, Theorem \ref{thm:AbsSpanMin} shows that every operator system with a finite-dimensional realization (see \cite{FNT+}) is completely normed by its finite dimensional boundary representations. For further material related to operator systems, completely positive maps, boundary representations, and the $C^*$-envelope we direct the reader to \cite{Ham79}, \cite{D96}, \cite{MS98}, \cite{F00}, \cite{F04}, \cite{FHL+}, and \cite{PSS+}.

\newpage

\section{Appendix}

The appendix contains an NC LDL$^*$ formula and the proof of Theorem \ref{theorem:EHKMRealComplex} over the reals. 

\subsection{The NC LDL$^*$ of block $3 \times 3$ matrices}
\label{sec:NCLDL}

This subsection contains a brief discussion of the NC LDL$^*$ decomposition of the evaluation of a linear pencil $L_A$ on a block $3 \times 3$ matrix. Consider a general block $3 \times 3$ tuple
\[
Z:=\begin{pmatrix}
X & \beta & \eta \\
\beta^* & \gamma & \sigma \\
\eta^* & \sigma^* & \psi
\end{pmatrix} 
\]
where $X \in \smnonekg$ and  $\gamma \in \smntwokg$ and $\psi \in \smnthreekg$ and $\beta, \eta,$ and $\sigma$ are each $g$-tuples of matrices of appropriate size. We know that
\[
L_A \begin{pmatrix}
X & \beta & \eta \\
\beta^* & \gamma & \sigma \\
\eta^* & \sigma^* & \psi
\end{pmatrix} \sim_{\mathrm{c.s.} }
\begin{pmatrix} 
L_A(X) & \Lambda_A (\beta) & \Lambda_A (\eta) \\
\Lambda_A (\beta^*) & L_A (\gamma) & \Lambda_A (\sigma) \\
\Lambda_A (\eta^*) & \Lambda_A (\sigma^*) & L_A (\psi)
\end{pmatrix}=:\mathfrak{Z}
\]
where $\sim_{\mathrm{c.s.}}$ denotes equivalence up to permutations (canonical shuffles). It follows that
\[
L_A (Z) \succeq 0 \mathrm{\ if \ and \ only \ if \ }
\mathfrak{Z} \succeq 0.
\]

The NC LDL$^*$ of $\mathfrak{Z}$ has as its block diagonal factor $D$ the matrix
\[
D=\begin{pmatrix}
L_A(X) & 0 & 0\\
0 & S & 0 \\
0 & 0 & L_A (\gamma)-\Lambda_A (\beta^*) L_A (X)^\dagger \Lambda_A (\beta)-W^* S^\dagger W
\end{pmatrix}
\]
where
\[
\begin{array}{rcl}
S&=& L_A (\psi)- \Lambda_A (\eta^*) L_A(X)^\dagger \Lambda_A (\eta) \\
W&=& \Lambda_A (\sigma^*)-\Lambda_A(\eta^*) L_A (X)^\dagger \Lambda_A (\beta).
\end{array}
\]
It follows that $L_A (Z) \succeq 0$ if and only if $L_A (X) \succeq 0$ and $S \succeq 0$ and
\[
L_A (\gamma)-\Lambda_A (\beta^*) L_A (X)^\dagger \Lambda_A (\beta)-W^* S^\dagger W \succeq 0.
\]

Considering the case where $\K=\R$ and $\gamma \in \R^g$ and $\psi=0 \in \R^g$, hence $\sigma=\sigma^* \in \R^g$, and substituting $\eta=c \hbeta$ or $\eta=0$ gives equations \eqref{eq:3LDL} and \eqref{eq:3LDLSig}, respectively.

\subsection{Proof of Theorem \ref{theorem:EHKMRealComplex} over the real numbers}
\label{sec:EHKMReal}

We now give a proof of Theorem \ref{theorem:EHKMRealComplex}  over the real numbers. To emphasize the real setting in this subsection we will now use the terms symmetric and orthogonal in favor of self-adjoint and unitary. Recall that a tuple $X \in \smnrg$ is irreducible over $\R$ if the matrices $X_1, \dots, X_g$ have no common reducing subspaces in $\R^n$; a tuple is reducible over $\R$ if it is not irreducible over $\R$. 

\begin{lem}
\label{lem:RealIrredSymCommutant}
Let $X \in \smnrg$ be a $g$-tuple of real symmetric matrices which is irreducible over $\R$ and let $W \in SM_n (\R)$ be a real symmetric matrix which commutes with $X$. Then $W$ is a constant multiple of the identity. 
\end{lem}
\begin{proof}
Let $W \in SM_n (\R)$ be a real symmetric matrix such that $WX=XW$ and let $\cE_1, \dots, \cE_k \subset \R^n$ denote the real eigenspaces of $W$ corresponding to the eigenvalues $\lambda_1, \dots, \lambda_k$ of $W$, respectively.  Since $X$ is real and $WX=XW$, each $\cE_j$ is a reducing subspace for $X$. If $k>2$, then each $\cE_j$ is a nontrivial real reducing subspace of $X$ which would imply that $X$ is reducible over $\R$. It follows that $k=1$ and $W= \lambda_1 I$. 
\end{proof}

We now prove  Theorem \ref{theorem:EHKMRealComplex} which is our real analogue of \cite[Theorem 1.1 (3)]{EHKM17}, Theorem \ref{theorem:EHKMRealComplex}. 

The proof over $\R$ follows exactly the proof over $\C$ in \cite{EHKM17} as we now outline. That an irreducible Arveson extreme point is absolute extreme is a simple argument given in \cite[Section 3.4]{EHKM17} based on \cite[Lemma 3.14]{EHKM17} which (over $\R$) says the following.

\begin{lemma}
 \label{lem:irreduciblegives}
Fix positive integer $n$ and $m$  and suppose  $C \in \R^{n \times m}$ is a nonzero matrix,  the tuple $X\in\smnrg$  is irreducible over $\R$  and $E\in \smmrg$.
If  $CX_j = E_j C$ for each $j,$  then $C^TC$
 is a nonzero multiple of
 the identity. Moreover, the range of
 $C$ reduces the set $\{E_1,\dots,E_g\}$ so there is an orthogonal matrix $U$ so that for each $j$ we have  $U^T E_j U= X\oplus Z_j$ for some $Z_j \in SM_k (\R),$ where $k=m-n$.
\end{lemma}

\begin{proof}
To prove this statement note that $
   X_j C^T = C^T E_j.
$
 It follows that
\[
 X_j C^TC = C^T E_j C = C^T C X_j.
\]
 Using Lemma \ref{lem:RealIrredSymCommutant} with the irreduciblity of $\{X_1,\dots,X_g\}$ shows $C^TC$ is a nonzero multiple
 of the identity and therefore  $C$ is a real multiple of an isometry.
 Further, since $CX=EC$, the range of $C$ is invariant for $E$.  Since each $E_j$
 is symmetric, the range of $C$ reduces each $E_j$ and $C$, as an isometric mapping
 into its range is a multiple of an orthogonal matrix.
\end{proof}

\noindent{\it Proof of Theorem \ref{theorem:EHKMRealComplex} when $\K=\R$.}

 Suppose $X$ is both irreducible over $\R$ and in the Arveson boundary of $\cDAR$.  To
  prove $X$ is an absolute extreme point, suppose
$
   X= \sum_{i=1}^\nu C_i^T E^i C_i,
$
 where each $C_i$ is nonzero, $\sum_{i=1}^\nu C_i^T C_i =I$ and $E^i \in \cDAR$. In this case, let
\[
   C= \begin{pmatrix} C_1 \\ \vdots \\ C_\nu \end{pmatrix}
\quad\text{ and }\quad
  E=
E^1\oplus E^2\oplus\cdots\oplus E^\nu
\]
 and observe that $C$ is an isometry and $X=C^T E C$.
   Hence, as $X$ is in the Arveson boundary,
  $CX=EC$. It follows that $C_i X_k = E^i_k C_i$ for each $i$ and $k$.  
  Thus, by Lemma \ref{lem:irreduciblegives},
  it follows that for each $i$ there is an orthogonal matrix $U_i$  such that $U_i^T E^i U_i= X\oplus Z^i$ for some $Z^i\in \cDAR$.
Therefore $X$ is an 
absolute
extreme point of $\cDAR$.

The converse proof that an absolute extreme point of $\cDAR$ is irreducible over $\R$ and Arveson extreme is \cite[Lemma 3.11]{EHKM17} and \cite[Lemma 3.13]{EHKM17} which while stated over $\C$ is unchanged over $\R$. \hfill \qedsymbol

\newpage

\centerline{ NOT FOR PUBLICATION}

\tableofcontents

\printindex


\begin{thebibliography}{1}

\bibitem[A88]{A88} J. Agler, {\it An abstract approach to model theory}, Surveys of some recent results in operator theory, Vol. II, 1-23, Pitman Res. Notes Math. Ser., 192, Longman Sci. Tech., Harlow, 1988.

\bibitem[A69]{A69} W. Arveson:
{\it Subalgebras of $C^*$-algebras}, Acta Math. {\bf 123} (1969) 141-224.

\bibitem[A72]{A72} W. Arveson:
{\it Subalgebras of $C^*$-algebras, II}, Acta Math. {\bf 128}  (1972) 271-308.

\bibitem[A08]{A08} W. Arveson: {\it The noncommutative Choquet boundary}, J. Amer. Math. Soc. {\bf 21} (2008) 1065-1084.

\bibitem[D96]{D96}
K.R. Davidson: {\it $C^*$-algebras by example}, American Mathematical Soc., 1996.

\bibitem[DK15]{DK15} K.R. Davidson, M. Kennedy:
{\it The Choquet boundary of an operator system}, Duke Math. J. {\bf 164} (2015) 2989-3004.

\bibitem[DM05]{DM05} M.A. Dritschel, S.A. McCullough:
{\it Boundary representations for families of representations of operator algebras and spaces}, J. Operator Theory {\bf 53} (2005) 159-168.

\bibitem[EW97]{EW97} E.G. Effros, S. Winkler: {\it Matrix convexity: operator analogues of the bipolar and Hahn-Banach theorems}, J. Funct. Anal. {\bf 144} (1997) 117-152.

\bibitem[E18]{E17} E. Evert: {\it Matrix convex sets without absolute extreme points}, Linear Algebra Appl. {\bf 537} (2018) 287-301.

\bibitem[EHKM18]{EHKM17}
E. Evert, J.W. Helton, I. Klep, S. McCullough:
{\it Extreme points of matrix convex sets, free spectrahedra and dilation theory},
J. of Geom. Anal. {\bf 28} (2018) 1373-1498.

\bibitem[F00]{F00} D.R. Farenick:
{\it Extremal matrix states on operator systems}, J. London Math. Soc. {\bf 61} (2000) 885-892.

\bibitem[F04]{F04} D.R. Farenick:
{\it Pure matrix states on operator systems}, Linear Algebra Appl. {\bf 393} (2004) 149-173.

\bibitem[FNT17]{FNT+} T. Fritz, T. Netzer, A. Thom: {\it Spectrahedral Containment and Operator Systems with Finite-dimensional Realization}, SIAM J. Appl. Algebra Geom. {\bf 1} (2017) 556-574.

\bibitem[FHL18]{FHL+}
A.H. Fuller, M. Hartz, M. Lupini:
{\it Boundary representations of operator spaces, and compact rectangular matrix convex sets}, J. Operator Theory {\bf 79} (2018) 139-172.

\bibitem[Ham79]{Ham79}
M. Hamana:
{\it Injective envelopes of operator systems},
 Publ. Res. Inst. Math. Sci. {\bf 15} (1979) 773-785.

\bibitem[HKM13]{HKM13}
J.W. Helton, I. Klep, S. McCullough: 
{\it The matricial relaxation of a linear matrix inequality}, Math. Program. {\bf 138} (2013) 401-445.

\bibitem[HM12]{HM12}
J.W. Helton, S. McCullough: {\it Every free basic convex semi-algebraic set has an LMI representation}, Ann. of Math. (2) {\bf 176} (2012) 979-1013.

\bibitem[HJ12]{HJ12}
R.A. Horn, C.R. Johnson:
\textit{Matrix analysis}, Cambridge university press, 2012.

\bibitem[KLS14]{KLS14}
C. Kleski: {\it Boundary representations and pure completely positive maps}, J. Operator Theory {\bf 71} (2014) 45-62.

\bibitem[K+]{K+}
T. Kriel: {\it Free spectrahedra, determinants of monic linear pencils and decompositions of pencils}, preprint
https://arxiv.org/abs/1611.03103.

\bibitem[MS98]{MS98}
P.S. Muhly, B. Solel:
"An algebraic characterization of boundary representations" In {\it Nonselfadjoint Operator Algebras, Operator Theory, and Related Topics}, Oper. Theory Adv. Appl. {\bf 104}, Birk\"{a}user, Basel, 1998, 189-196.

\bibitem[PSS18]{PSS+}
B. Passer, O. Shalit, B. Solel:
{\it Minimal and maximal matrix convex sets}, J. Funct. Anal. {\bf 274} (2018) 3197-3253.

\bibitem[P02]{P02}
V. Paulsen: {\it Completely bounded maps and operator algebras}, Cambridge Studies in Advanced Mathematics 78, Cambridge University Press, 2002.

\bibitem[WW99]{WW99}
C. Webster and S. Winkler:
{\it The Krein-Milman Theorem in Operator Convexity},  Trans Amer. Math. Soc. {\bf 351} (1999) 307-322.

\bibitem[Z17]{Z17} 
A. Zalar:
{\it Operator Positivstellen\"{a}tze for noncommutative polynomials positive on matrix convex sets}, J. Math. Anal. Appl. {\bf 445} (2017) 32-80.

\end{thebibliography}
\end{document}